\documentclass[12pt]{amsart}
\usepackage{epsf}

\usepackage{verbatim}
\usepackage{amsmath}
\usepackage{amsopn}
\usepackage{amssymb}
\usepackage{latexsym}
\usepackage{amsfonts}
\usepackage{amsthm}

\newcommand{\rh}{\frac{r}{2}}

\newtheorem{thm}{Theorem}[section]
\newtheorem{pro}[thm]{Proposition}
\newtheorem{lem}[thm]{Lemma}
\newtheorem{cla}[thm]{Claim}

\newtheorem{cor}[thm]{Corollary}

\theoremstyle{definition}
\newtheorem{obs}[thm]{Observation}
\newtheorem{pres}[thm]{Presentation}

\newtheorem{rem}[thm]{Remark}
\newtheorem{exa}[thm]{Example}
\newtheorem{defn}[thm]{Definition}

\newtheorem{prop}[thm]{Proposition}
\newtheorem{conj}[thm]{Conjecture}

\newcommand{\sumlim}{\sum\limits}
\newcommand{\prodlim}{\prod\limits}
\newcommand{\een}{\end{enumerate}}
\newcommand{\blem}{\begin{lem}}
\newcommand{\elem}{\end{lem}}
\newcommand{\bcl}{\begin{cla}}
\newcommand{\ecl}{\end{cla}}
\newcommand{\ethm}{\end{thm}}
\newcommand{\bpr}{\begin{pro}}
\newcommand{\epr}{\end{pro}}
\newcommand{\bco}{\begin{cor}}
\newcommand{\eco}{\end{cor}}
\newcommand{\bcon}{\begin{conj}}
\newcommand{\econ}{\end{conj}}
\newcommand{\bde}{\begin{defn}}
\newcommand{\ede}{\end{defn}}
\newcommand{\bex}{\begin{exa}}
\newcommand{\eexa}{\end{exa}}
\newcommand{\bobs}{\begin{obs}}
\newcommand{\eobs}{\end{obs}}
\newcommand{\bexe}{\begin{exe}}
\newcommand{\eexe}{\end{exe}}

\newcommand{\Z}{{\Bbb Z}}

\newcommand{\grn}{G_{r,n}}
\newcommand{\arn}{A_{r,n}}

\begin{document}
\title[Group of alternating colored permutations]{On the group of alternating colored permutations}
\author{Eli Bagno, David Garber and Toufik Mansour}

\address{Eli Bagno, The Jerusalem College of Technology, Jerusalem, Israel}
\email{bagnoe@jct.ac.il}

\address{David Garber, Department of Applied Mathematics, Faculty of
  Sciences, Holon Institute of Technology, 52 Golomb St., PO
  Box 305, 58102 Holon, Israel}
\email{garber@hit.ac.il}

\address{Toufik Mansour, Department of Mathematics, University of Haifa, 31905 Haifa,
Israel.}
\email{toufik@math.haifa.ac.il}

\begin{abstract}
The group of alternating colored permutations is the natural analogue of the classical alternating group, inside the wreath product $\mathbb{Z}_r \wr S_n$.
We present a 'Coxeter-like' presentation for this group and compute the length function with respect to that presentation.
Then, we present this group as a covering of $\mathbb{Z}_{\rh} \wr S_n$ and use this point of view to give another expression for the length function. We also use this covering to lift several known parameters of $\mathbb{Z}_{\rh} \wr S_n$ to the group of alternating colored permutations.
\end{abstract}

\date{\today}

\maketitle

\section{Introduction}

The group of colored permutations, $\grn$, is a natural generalization of the Coxeter groups of types A (the symmetric group) and B (the hyperoctahedral group). Extensive research has been devoted to extending the enumerative combinatorics aspects and methods from the symmetric group to the group of colored permutations (see for example \cite{B,BG,DF,NR,RR2,S}, and many more).

It is well-known that the symmetric group $S_n$ has a system of Coxeter generators which consists of the adjacent transpositions: 
$$\{(i,i+1)\mid 1 \leq i \leq n-1\}.$$

The alternating subgroup, $A_n$, which is the kernel of the sign homomorphism, is a well-known subgroup of the symmetric group of index $2$. A pioneering work, expanding one of the fascinating branches of enumerative combinatorics, namely, the study of permutation statistics to $A_n$, has been done by Roichman and Regev in \cite{RR}. They defined some natural statistics which equidistribute over $A_n$ and yielded identities for their generating functions.

Brenti, Reiner, and Roichman \cite{BRR} dealt with the alternating subgroup of an arbitrary Coxeter group.  They started by exploring Bourbaki's presentation \cite[Chap. IV, Sec. 1, Exer. 9]{bourbaki} and elaborated on a huge spectrum of extensions of the permutation statistics of $S_n$ to the (general) alternating group.

\medskip

In this paper, we study the subgroup of $\grn$, consisting of what we call {\em alternating colored permutations}, which is the analogue of the usual alternating group $A_n$ in the colored permutation group.  For every $n \in \mathbb{N}$ and even $r$, the mapping which sends all 'Coxeter-like' generators of $\grn$ (see the definition in Section \ref{pre}) to $-1$ is a $\mathbb{Z}_2$-character, whose kernel is what we call here {\it the group of alternating colored permutations}, denoted by $A_{r,n}$.
We present here a generalization of Bourbaki's  presentation, for $r=4k+2$, equipped with a set of canonical words, an algorithm to find a canonical presentation for each element of the group, and a combinatorial length function.

For the study of permutation statistics of $A_n$, Regev and Roichman \cite{RR} used a covering map from $A_{n+1}$ to $S_n$, which enabled them to pass parameters from $S_n$ to the alternating group $A_{n+1}$. In this paper, we use a similar idea, where in this time we consider the group of alternating colored permutations as a $2^{n-1}$-cover of the group of colored permutations of half the number of colors. We use this technique to shed a combinatorial flavor on
our length function and to pass some statistics and their generating function to the group of alternating colored permutations.

Note that there are two additional candidates for the group of alternating colored permutations. Namely, every $\mathbb{Z}_2$-character of $\grn$ provides a kernel which deserves to be called a group of alternating colored permutations.
A work in this direction which gives a profound treatment to the other two non-trivial kernels, and points out the connections between the three groups, and some interesting properties of each group separately is in progress.

\medskip

This paper is organized as follows. In Section \ref{pre}, we gather the needed definitions on the colored permutation group, as well as some notations which we use in the sequel. A Coxeter-like presentation for the group of colored permutations, $\grn$, is presented  at the end of this section. The notion of alternating colored permutations is introduced in Section \ref{arn}. We present its set of generators, and show their corresponding relations. In Section \ref{alg}, we present an algorithm for writing each element as a product of the generators. A detailed analysis of that algorithm yields a set of canonical words, as well as a length function.
Section \ref{proofs} is devoted to some technical proofs, as well as to the generating function of the length function.

In Section \ref{cover}, we present the covering map and study the structure of the cosets, thereby providing a way to decompose the length function via the quotient group. The part of the length which varies over each coset (fiber) is called {\it the fibral length} and is studied here in a combinatorial way. Then we provide a generating function for this parameter. In Section \ref{perm stat}, we give some examples for using the covering map for lifting parameters from the colored permutations group of half the number of colors to the group of alternating colored permutations.

\section{Preliminaries and notations}\label{pre}
In this section, we gathered some notations as well as preliminary notions which will be needed for the rest of the paper.

\subsection{The group of colored permutations}
\bde
Let $r$ and $n$ be positive integers. {\it The group of colored
permutations of $n$ digits with $r$ colors} is the wreath product
$$\grn=\mathbb{Z}_r \wr S_n=\mathbb{Z}_r^n \rtimes S_n,$$
consisting of all pairs $(\vec{z},\tau)$, where $\vec{z}$ is an $n$-tuple of
integers between $0$ and $r-1$ and $\tau \in S_n$. The
multiplication is defined by the following rule: for
$\vec{z}=(z_1,\dots,z_n)$ and $\vec{z'}=(z'_1,\dots,z'_n)$,
\begin{equation}
(\vec{z},\tau) \cdot (\vec{z'},\tau')=((z_1+z'_{\tau^{-1}(1)},\dots,z_n+z'_{\tau^{-1}(n)}),\tau \circ \tau')
\end{equation}
(here $+$ is taken modulo $r$).
\ede

Here is another way to present $\grn$: Consider the alphabet
$$\Sigma=\{1,\dots,n,\bar{1},\dots,\bar{n},\dots,
1^{[r-1]},\dots,n^{[r-1]} \}$$ as the set $[n]$ colored by the
colors $0,\dots,r-1$. Then, an element of $\grn$ is a {\it colored
permutation}, i.e., a bijection $\pi: \Sigma \rightarrow \Sigma$
satisfying the following condition: if
$\pi \left( i^{[\alpha]} \right)=j^{[\beta]}$, then
$\pi \left( i^{[\alpha+1]} \right)=j^{[\beta+1]}$  (the addition in the exponents is taken modulo $r$). Using this approach, the element
$\pi=((z_1,\dots,z_n),\tau) \in \grn$ is the permutation on $\Sigma$, satisfying  $\pi(i)=\pi(i^{[0]})=\tau(i)^{[z_{\tau(i)}]}$ for each $1 \leq i \leq n$.
For example, the element $\pi=\left((2,1,0,3),\begin{pmatrix} 1 & 2 & 3 &4 \\
2 & 1 & 4 &3
\end{pmatrix}\right) \in G_{3,4}$ satisfies: $\pi(1)=2^{[1]},\pi(2)=1^{[2]},\pi(3)=4^{[3]},\pi(4)=3^{[0]}$.

For an element $\pi=(\vec{z},\tau) \in \grn$ with $\vec{z}=(z_1,\dots,z_n)$, we
write $z_i(\pi)=z_i$, and  denote
$|\pi|=(\vec{0},\tau)$. We define also $c_i(\pi)=r-z_i(\pi^{-1})$ and $\vec{c}(\pi)=\vec{c}=(c_1,\dots,c_n)$. Using this notation, the element $\pi=(\vec{z},\tau)=\left((2,1,0,3),\begin{pmatrix} 1 & 2 & 3 &4 \\
2 & 1 & 4 &3
\end{pmatrix}\right)$ satisfies $\vec{c}=(1,2,3,0))$.

We usually write $\pi$ in its {\it window notation} (or {\it one line notation}):
$\pi=\left( a_1^{[c_1]} \cdots a_n^{[c_n]} \right)$, where $a_i=\tau(i)$, so in our example we have:
$\pi=(2^{[1]} 1^{[2]} 4^{[3]} 3^{[0]})$ or just  $\left(\bar{2} \bar{\bar{1}}
\bar{\bar{\bar{4}}} 3\right).$

Note that $z_i$ is the color of the digit $i$ ($i$ is taken from the window notation), while $c_j$ is the color of the digit $\tau(j)$.
Here, $j$ stands for the place, whence $i$ stands for the value.

The group $\grn$ is generated by the set of generators ${\mathcal S}=\{s_0,s_1,\dots, s_{n-1}\}$, defined by their action on the set $\{1,\dots,n\}$ as follows:
$$s_i(j)=\left\{ \begin{array}{ccc} i+1 & j=i \\
i & j=i+1 \\
j & \text{otherwise,}
\end{array}    \right.$$
whereas the generator $s_0$ is defined by
$$s_0(j)= \left\{ \begin{array}{cc}
\bar{1} & j=1 \\
j & \text{otherwise.}
\end{array}    \right.$$

It is easy to see that the group $\grn$ has the following 'Coxeter-like' presentation with respect to the set of generators ${\mathcal S}$:

\begin{pres}\label{presentation of grn}
\
\begin{itemize}
\item $s_0^r=1$,
\item $s_i^2=1 \mbox{ for }  1 \leq i \leq n-1$,
\item $s_i s_{i+1} s_i = s_{i+1} s_i s_{i+1} \mbox{ for } 1 \leq i <n$,
\item $s_i s_j= s_j s_i \mbox{ for } 1 \leq i <j<n, \ j-i>1$,
\item $(s_0 s_1)^{2r}=1.$
\end{itemize}
\end{pres}

\subsection{Some permutation statistics}

For $\pi \in \grn$, define the \emph{length} of $\pi$ with
respect to the set of generators $\mathcal S$ to be the minimal number of
generators whose product is $\pi$. Formally:
$$\ell(\pi)=\min\{r \in \mathbb{N}: \pi=s_{i_1} \cdots s_{i_r},
\mbox{for } i_1,\dots,i_r \in \{0,\dots,n-1\} \ \}.$$
\begin{defn}
The {\it length order} on the alphabet
$$\Sigma=\{1,\dots,n,\bar{1},\dots,\bar{n},\dots,1^{[r-1]},\dots,n^{[r-1]}\}$$
is defined as follows:

\begin{equation}
n^{[r-1]}< \cdots <\bar{n}<\cdots <1^{[r-1]}< \cdots <\bar{1}<1< \cdots <n
\end{equation}

Let $\sigma \in \grn$. We define:
$${\rm csum}(\sigma) = \sumlim_{i=1}^n c_i(\sigma)=\sumlim_{i=1}^n z_i(\sigma).$$
For $\pi \in G_{r,n}$, the {\it inversion number}, ${\rm inv}(\pi)$, is defined  as follows:
$${\rm inv}(\pi)=|\{(i,j) \mid i<j, \pi(i)>\pi(j) \} |.$$
where the partial order is the length order defined above.
\end{defn}

For any $a,n \in \mathbb{N}$, let $R_n(a)$ be the representative of $[a] \in \mathbb{Z}_n$ satisfying $0 \leq a <n$.

In the sequel, we will use the following operator:

\begin{defn}
 Let $a \in \mathbb{N}$.
\begin{equation} \label{oslash}
a\oslash 2=\left\{ \begin{array}{cc} R_{\rh}(\frac{a}{2}) & a \equiv 0 ({\rm mod}\ 2) \\
                                                                             R_{\rh}(\frac{a+\frac{r}{2}}{2}) &  a \not\equiv 0 ({\rm mod}\ 2)\\               \end{array}    \right.
\end{equation}
\end{defn}
It is easy to see that the operator $\oslash$ commutes with the addition operation in
$\mathbb{Z}_{\frac{r}{2}}$, i.e.
\begin{equation} \label{oslash commutes}
((a+b)\oslash 2) \equiv ((a\oslash 2) + (b \oslash 2)) \left({\rm mod}\ {\rh} \right)
\end{equation}

\section{The group of alternating colored permutations}\label{arn}

The main target of this paper is the group of alternating colored permutations. We proceed now to its definition. Let $\varphi$ be the function defined on the set $\mathcal{S}$ by
$\varphi(s_i)=-1$ for any  $0 \leq i \leq n-1$. It is easy to see
that for even $r$, $\varphi$ can be uniquely extended to a
homomorphism from $\grn$ to $\mathbb{Z}_2$, so the following is well-defined:

\bde
Let $r$ be an even positive number. Define:
$$A_{r,n}=\ker(\varphi).$$ The group $A_{r,n}$ is called the {\em alternating
subgroup of $G_{r,n}$}.
\ede

Since $A_{r,n}$ is a subgroup of index $2$, we have:
$|A_{r,n}|=\frac{r^n n!}{2}$.

\medskip

In this paper, we concentrate on the case $r=4k+2$. The other case
will be treated in a subsequent paper.

\medskip

We start by presenting a set of generators for $A_{r,n}$ (we prove that they indeed
generate the group in the next section). Define:
$${\mathcal A}=\{a_0,a_1,a_1^{-1},a_2,\dots,a_{n-1}\},$$
where:
$$\begin{array}{l}
a_i=s_0^{\frac{r}{2}}s_i \ \ \mbox{ for } 1 \leq i \leq n-1 \\ a_0=s_0^2.
\end{array}$$

It is easy to see that the following translation relations hold in $\grn$:
\begin{enumerate}
\item $s_i s_j = a_i a_j \mbox{ for }  i,j \in\{2,\dots, n-1\}$,
\item $s_1 s_i = a_1^{-1} a_i \mbox{ for } i \in\{2,\dots, n-1\}$,
\item $s_i s_1 = a_i a_1 \mbox{ for } i \geq 2$,
\item $s_0 s_1 = a_0^{\frac{r+2}{4}} a_1 $,
\item $s_1 s_0 = a_1^{-1} a_0^{\frac{r+2}{4}}$,
\item $a_0^{\frac{r}{2}} = 1 $,
\item $s_0 s_i = a_0^{\frac{r+2}{4}} a_i$.
\end{enumerate}

\section{The Combinatorial algorithm}\label{alg}
In this section, we introduce an algorithm which presents each element of $A_{r,n}$ as a product of
the set of generators $\mathcal A$ of $A_{r,n}$ in a canonical way.

Let $\pi \in A_{r,n}$. We first refer to $\pi$ as an element of $G_{r,n}$ and apply the known algorithm on $\pi$ to write it as a product of elements in $\mathcal{S}$. In the second step, we
translate that presentation into the set of generators $\mathcal A$ of $A_{r,n}$.

The algorithm for writing $\pi$ as a product of elements in $\mathcal{S}$ consists of two parts: the {\it coloring part} and the {\it ordering part}.

In the coloring part, we
start from the identity element and color all the digits $i$ having $z_i
\neq 0$. This part terminates with an ordered permutation $\sigma$
with respect to the length order. In the second part, we use only
generators of the set $\mathcal{S}-\{s_0\}$ to arrive at $\pi$ from the ordered permutation $\sigma$.

\subsection{The coloring part}\label{color order}
Define:
$${\rm Col}(\pi)=\{1 \leq i \leq n \mid z_i(\pi) \neq 0\},$$ and
$${\rm col}(\pi)=|{\rm Col}(\pi)|.$$
Note that the set ${\rm Col}(\pi)$ contains the colored digits in the image of $\pi$, (i.e. those appearing in the window notation), and not their places.
We order ${\rm Col}(\pi)$ as follows: ${\rm Col}(\pi)=\{i_1<i_2< \cdots
<i_{{\rm col}(\pi)}\}$.

We start with the identity element and color each digit $i \in {\rm Col}(\pi)$ by $z_i$ colors. This process is done according to the order of the elements in ${\rm Col}(\pi)$. We use $s_{{i_k}-1}s_{i_k-2} \cdots s_1s_0^z$ to color the digit $i_k$ by $z$ colors.

\begin{exa}
Let $\pi= \left( 1 2^{[2]} 4 5^{[1]} 3^{[3]} \right) \in G_{6,5}$.
$$ (1 2 3 4 5) \stackrel{s_1s_0^2}{\rightarrow}  \
\left( 2^{[2]} 1 3 4 5 \right)  \stackrel{s_2 s_1 s_0^3}{\rightarrow}  \ \left( 3^{[3]} 2^{[2]} 1 4 5 \right)
\stackrel{s_4s_3s_2s_1s_0}{\rightarrow}  \ \left( 5^{[1]} 3^{[3]} 2^{[2]} 1 4 \right) =\sigma.$$
The permutation $\sigma$ is an ordered permutation with respect to the length order.
\end{exa}

\subsection{The ordering part}
For simplifying the presentation, in this part we start with $\pi$ and arrive at the ordered permutation $\sigma$, instead of continuing the algorithm from the point we have left it at the end of the coloring part.

We start by pushing the element $i_1=|\sigma|(1)$ in the window notation of $\pi$ to its
correct place. Let $p=|\pi|^{-1}(i_1)$. The pushing is
done by multiplying $\pi$ (from the right) by the element $u_1=s_{p-1}s_{p-2}\cdots s_1$.

Now, we continue to push the other digits of $\pi$:  For each $1 < k \leq n-2$, assuming that $i_k=|\sigma|(k)$ is now located at position
$p$, we use the element $s_{p-1}s_{p-2}\cdots s_k$ in order to push the digit $i_k$ to its
correct place.

\begin{exa}
We continue the previous example. Again, let\break $\pi= \left( 1 2^{[2]} 4 5^{[1]} 3^{[3]} \right)$.
The coloring part ends with the following ordered permutation:
$$ \sigma=  \ \left( 5^{[1]} 3^{[3]} 2^{[2]} 1 4 \right). $$

Now, we go the other way around: we start with $\pi$ and order it until we reach
$\sigma$:
$$\pi= \left( 1 2^{[2]} 4 5^{[1]} 3^{[3]} \right) \stackrel{s_3s_2 s_1}{\rightarrow}
\left( 5^{[1]} 1 2^{[2]} 4 3^{[3]} \right)  \stackrel{s_4 s_3
s_2 }{\rightarrow} $$
$$ \rightarrow \left( 5^{[1]} 3^{[3]} 1 2^{[2]} 4 \right) \stackrel{s_3
}{\rightarrow}  \left( 5^{[1]} 3^{[3]} 2^{[2]} 1 4 \right) =\sigma.
$$

Therefore, we have:
$$\pi=\underbrace{s_1 s_0^2 \cdot s_2 s_1 s_0^3 \cdot s_4 s_3 s_2 s_1 s_0}_{\rm coloring\ part} \cdot \underbrace{s_3 \cdot s_2 s_3 s_4 \cdot s_1 s_2 s_3}_{\rm ordering\ part}.$$
\end{exa}

The algorithm described above gives a {\it reduced} word representing $\pi$ in the generators of $G_{r,n}$. This fact was proved in \cite[Theorem 4.3]{B}. The same algorithm can also be found in \cite{S}; see also \cite{RR2}. The word which was obtained in this way is called the {\it canonical decomposition} of $\pi$.

\subsection{Translation}\label{Translation}
Now, we translate the word obtained by the algorithm described above into a word in the generators in $\mathcal A$:
Let $\pi \in A_{r,n}$. Use the above algorithm to write a reduced expression of $\pi$ (in the usual generators of $\grn$) in the form:
$s_{i_1}s_{i_2} \cdots s_{i_{2k}}$.
Divide the elements of the reduced expression into pairs: $(s_{i_1}s_{i_2}) \cdots (s_{i_{2k-1}}s_{i_{2k}})$. Now, insert $s_0^{\rh} s_0^{\rh}$ inside each pair, as follows:
\begin{eqnarray*}
\left( s_{i_1}s_0^{\rh} s_0^{\rh}s_{i_2} \right) \cdots \left( s_{i_{2k-1}}s_0^{\rh} s_0^{\rh}s_{i_{2k}} \right) & = & \left( s_{i_1}s_0^{\rh} \right) \left( s_0^{\rh}s_{i_2} \right) \cdots \left( s_{i_{2k-1}}s_0^{\rh} \right)  \left( s_0^{\rh}s_{i_{2k}} \right) =\\
& = & a_{i_1}^{\varepsilon_{i_1}} \cdots a_{i_{2k}}^{\varepsilon_{i_{2k}}},
\end{eqnarray*}
where $\varepsilon_{i_j}=1$ if $i_j>1$, $\varepsilon_{i_j} \in \{\pm 1\}$ if $i_j=1$, and $\varepsilon_{i_j} \in \left\{ 1,\dots , \rh-1 \right\}$ if $i_j=0$.

\begin{exa}\label{pairing}
We continue with $\pi=\left( 1 2^{[2]} 4 5^{[1]} 3^{[3]} \right) \in A_{6,5}$ from the previous examples.
As we saw, $\pi=s_1 s_0^2 s_2 s_1 s_0^3 s_4 s_3 s_2 s_1 s_0 s_3 s_2 s_3 s_4 s_1 s_2 s_3$. Now, we perform the translation:
\begin{eqnarray*} \pi & = & (s_1 s_0) (s_0 s_2) (s_1 s_0) (s_0 s_0) (s_4 s_3) (s_2 s_1) (s_0 s_3) (s_2 s_3) (s_4 s_1) (s_2 s_3)=\\
                      & = & (s_1 s_0^3) (s_0^3 s_0) (s_0 s_0^3) (s_0^3 s_2) (s_1 s_0^3) (s_0^3 s_0) (s_0 s_0^3) (s_0^3s_0) (s_4s_0^3) (s_0^3 s_3) \cdot \\
                      & {} &  \cdot (s_2 s_0^3) (s_0^3 s_1) (s_0s_0^3) (s_0^3 s_3) (s_2s_0^3) (s_0^3 s_3) (s_4s_0^3) (s_0^3 s_1)(s_2 s_0^3) (s_0^3s_3)=\\
                      & = & a_1^{-1} \bold{a_0^2 a_0^2} a_2 a_1^{-1} \bold{a_0^2 a_0^2 a_0^2} a_4 a_3 a_2 a_1 a_0^2 a_3 a_2 a_3 a_4 a_1 a_2 a_3 =\\
                      & =& a_1^{-1} a_0 a_2 a_1^{-1}a_4 a_3 a_2 a_1 a_0^2 a_3 a_2 a_3 a_4 a_1 a_2 a_3.
\end{eqnarray*}
In the last equality, we cancelled some appearances of the bold-faced generator $a_0$, since in $A_{6,5}$, $a_0^3=1$.
\end{exa}

\subsection{Analysis of the algorithm}

For analyzing the algorithm described above, we define the following sets of elements of $\grn$ and $A_{r,n}$.

\subsubsection{The coloring part}
Let
$$C_1=\left\{ 1,s_0^2,s_0^4 \dots ,s_0^{r-2} \right\}=\left\{ 1,a_0,\dots,a_0^{\frac{r-2}{2}} \right\}.$$

For each $1 < i \leq n$, define for odd $i-1$:
\begin{eqnarray*}
C_i^0 & = & \left\{s_0s_{i-1} \cdots s_1,s_0s_{i-1} \cdots s_1s_0^2,s_0s_{i-1} \cdots s_1s_0^4,\dots,s_0s_{i-1} \cdots s_1s_0^{r-2} \right\}\\
C_i^1 & = & \left\{ 1,s_{i-1} \cdots s_1s_0,s_{i-1}\cdots s_1s_0^3,\dots,s_is_{i-1} \cdots s_1s_0^{r-1} \right\},
\end{eqnarray*}
or, in the language of the set $\mathcal A$ of generators of $A_{r,n}$:
\begin{eqnarray*}
C_i^0 & = & \left\{ \begin{array}{l} a_0^{\frac{r+2}{4}}a_{i-1} \cdots a_1,a_0^{\frac{r+2}{4}}a_{i-1} \cdots a_1a_0,a_0^{\frac{r+2}{4}}a_{i-1} \cdots a_1a_0^2,\dots , \\ a_0^{\frac{r+2}{4}}a_{i-1} \cdots a_1a_0^{\rh-1} \end{array} \right\} \\
C_i^1 & = & \left\{ 1,a_{i-1} \cdots a_1^{-1},a_{i-1} \cdots a_2a_1^{-1}a_0,\dots,a_{i-1} \cdots a_1^{-1}a_0^{\frac{r}{2}-1} \right\}.
\end{eqnarray*}


For even $i-1$, we define:
\begin{eqnarray*}
C_i^{0} & = & \left\{ s_0 s_{i-1} \cdots s_1s_0,s_0s_{i-1} \cdots s_1s_0^3,\dots,s_0s_{i-1} \cdots s_1s_0^{r-1} \right\}\\
C_i^1 & = & \left\{ 1,s_{i-1} \cdots s_1,s_{i-1} \cdots s_1 s_0^2,s_{i-1} \cdots s_1 s_0^4,\dots,s_{i-1} \cdots s_1s_0^{r-2} \right\},
\end{eqnarray*}
or, in the language of the set $\mathcal A$ of generators of $A_{r,n}$:
\begin{eqnarray*}
C_i^0 & = & \left\{ \begin{array}{l} a_0^{\frac{r+2}{4}}a_{i-1}  \cdots a_1^{-1},a_0^{\frac{r+2}{4}}a_{i-1} \cdots a_1^{-1}a_0,
a_0^{\frac{r+2}{4}}a_{i-1} \cdots a_1^{-1}a_0^2,\dots ,\\ a_0^{\frac{r+2}{4}}a_{i-1} \cdots a_1^{-1}a_0^{\rh-1} \end{array} \right\}\\
C_i^1 & = & \left\{ 1,a_{i-1} \cdots a_1, a_{i-1} \cdots a_1a_0,a_{i-1} \cdots a_1a_0^2,\dots,a_{i-1} \cdots a_1a_0^{\rh-1} \right\}.
\end{eqnarray*}

Define also:
$$C_i = C_i^0 \cup C_i^1 \mbox{ and } C_{n+1}=\left\{ 1,a_0^{\frac{r+2}{4}} \right\}.$$

Let $\pi \in A_{r,n}$. Write $\pi$ as a product of the generators of $G_{r,n}$ in the canonical form described above.

If there is no coloring part, then $\pi \in A_n$ (the classical alternating group in $S_n$), so its expression contains an even number of generators from the set $\{s_1,\dots,s_{n-1}\}$. We can easily make the pairing by the relations mentioned above. The length of such an expression is clearly ${\rm inv}(\pi)$ (note that in this case, it does not matter whether we use the length order or the usual order).

Otherwise, we start with the coloring part. Denote by $i=i_1$ the smallest colored digit in the window notation of $\pi$, and by $z=z_{i_1}$ its color. We divide our treatment into four cases:

\begin{enumerate}
\item {\bf $i-1$ and $z$ are  both even}: In this case, we translate $s_{i-1} \cdots s_1$ to $a_{i-1} \cdots a_1$ and $s_0^{z}$ to $a_0^{\frac{z}{2}}$, so the contribution of this sub-expression is $i-1+R_{\rh}(\frac{z}{2})=i-1+R_{\rh}(z \oslash{2})$. We have used $a_{i-1} \cdots a_1 a_0^{\frac{z}{2}} \in C_i$.

\item {\bf $i-1$ is even and $z$ is odd}: \label{i-1 even zi odd} In this case, we translate $s_{i-1} \cdots s_1 s_0^{z-1}$ to $a_{i-1} \cdots a_1 a_0 ^{\frac{z-1}{2}} \in C_i$ and leave an additional generator $s_0$ which will be treated during the coloring of the next digit, or just before the ordering part. Note that since $\pi \in A_{r,n}$, there must be some $s_j$, $j \neq 0$, appearing right after the sub-expression $s_{i-1} \cdots s_1 s_0^{z}$.
    In calculating the contribution of coloring the current digit (including the missing generator $s_0$ which will be paired later), consider the sub-expression $$s_0^{z}s_j=s_0^{\rh}s_0^{z+\rh} s_j=a_0^{\frac{z+\rh}{2}} s_0^{\rh} s_j=a_0^{\frac{z+\rh}{2}} a_j=a_0^{z \oslash 2} a_j.$$
    Hence, $i$ contributes $i-1+R_{\rh}(z \oslash 2)$. We have used
    $$a_{i-1} \cdots a_1 a_0^{\frac{z-1}{2}} \in C_i,$$
    and note that in the next colored digit, we complete the remaining $a_0^{\frac{r+2}{4}}$ since $z \oslash 2=\frac{z-1}{2}+\frac{r+2}{4}$. If $i$ is the last colored digit, then the term $a_0^{\frac{r+2}{4}}\in C_{n+1}$ will be chosen from the set $C_{n+1}$.

\item {\bf $i-1$ and $z$ are both odd}: In this case, the sub-expression $s_{i-1} \cdots s_2$ will be translated to $a_{i-1} \cdots a_{2}$ and $s_1 s_0^{z}$ will be written as  $s_1 s_0^{\rh} s_0^{z+\rh}=a_1^{-1}a_0^{z \oslash 2}$. This expression contributes $i-1+R_{\rh}(z \oslash 2)$ to the length of $\pi$, and we have used
    $$a_{i-1} \cdots a_2 a_1^{-1}a_0^{z \oslash 2} \in C_i.$$

\item {\bf $i-1$ is odd and $z$ is even}: Here, again, the sub-expression $s_{i-1}\cdots s_2$ will be translated to $a_{i-1} \cdots a_2$ and $s_1s_0^{z-1}$ will be translated to $a_1^{-1}a_0^{\frac{z-1+\rh}{2}}=a_1^{-1}a_0^{(z-1) \oslash 2}$, so we use:
    $$a_{i-1} \cdots a_2 a_1^{-1} a_0^{(z-1) \oslash 2} \in C_{i},$$
    and leave an additional generator $s_0$ which will be paired with some $s_j$ during the coloring of the next digit or just before the ordering part.
    In order to calculate the contribution of coloring this digit to the length of $\pi$ (including the missing generator $s_0$ which will be paired later), we borrow the generator $s_j$ appearing just after the coloring expression of the current digit: $s_1s_0^{z}s_j=s_1s_0^{\rh}s_0^{z+\rh}s_j$. Since we wrote:  $s_1s_0^{\rh}=a_1^{-1}$, we are left with $s_0^{z+\rh}s_j=a_0^{z \oslash 2}$. The contribution in this case is again $i-1+R_{\rh}(z \oslash 2)$. Now, since
    $$z \oslash 2 \equiv \left( (z-1) \oslash 2+\frac{r+2}{4} \right) \left( {\rm mod} \ {\frac{r}{2}} \right),$$ we take
    $$a_{i-1} \cdots a_1^{-1} a_0^{(z-1) \oslash 2} \in C_i$$
    and the remaining $a_0^{\frac{r+2}{4}}$ will be taken from the next colored digit or from $C_{n+1}$ (as in case (2)).
\end{enumerate}

\medskip

Now, we apply the same procedure to the next colored digits, but note that there might be a situation in which the expression coloring the digit $j$ is $s_0s_{j-1}s_{j-2} \cdots$, due to our debt of the generator $s_0$ from the preceding colored digit, so the cases might be switched after converting $s_0 s_j$ to $a_0^{\frac{r+2}{4}}a_{j-1}$.

\medskip

The following example will illuminate the situation.

\begin{exa}
Let $\pi= \left( 1 2^{[2]} 4 5^{[1]} 3^{[3]} \right) \in A_{6,5}$.
Then the ordered permutation is: $\sigma=\left( 5^{[1]} 3^{[3]} 2^{[2]} 1 4 \right)$.
We perform the coloring part:
$$(1 2 3 4 5) \stackrel{\textcircled{1}\ s_1s_0}{\longrightarrow} \left( 2^{[1]} 1 3 5 4 \right)  \stackrel{\textcircled{2} \ s_0 s_2 s_1 s_0^3}{\longrightarrow}
\left( 3^{[3]} 2^{[2]} 1 4 5 \right) \stackrel{\textcircled{3} \ s_4s_3s_2s_1}{\longrightarrow} \left( 5 3^{[3]} 2^{[2]} 1 4 \right) = \sigma .$$

\noindent
{\bf Step \textcircled{1}}: The smallest colored digit is $2$, which has to be colored by two colors, so we are in case (4). We choose $s_1s_0=a_1^{-1}a_0^{2}$ from $C_2^1$. The additional generator $s_0$ will be treated in the next step. Note that in the calculation of the contribution of this step to the length of $\pi$ we borrow the generator $s_2$ from the next colored digit: 
$$s_1 s_0^2 s_2=s_1 s_0^{3+2+3} s_2 = \left( s_1 s_0^3 \right) s_0^2 \left( s_0^3 s_2 \right) = a_1^{-1} a_0 a_2.$$
This expression contributes only $2$ to the length of $\pi$. The generator $a_2$ will be counted in the next step.

\medskip

\noindent
{\bf Step \textcircled{2}}: The next colored digit is $3$, and we have a debt of a generator $s_0$ from the previous step. Thus, we choose
$$s_0 s_2 s_1 s_0 s_0 s_0 = s_0^4 \left( s_0^3 s_2 \right) \left( s_1 s_0^3 \right) = a_0^2 a_2 a_1^{-1} \in C_3^0.$$
Even though $i-1=2$ is even and $z=3$ is odd (case (2)), after $s_0s_2$ in the previous step, we are actually again in case (4).
Note that the expressions $a_1^{-1} a_0^2$ from step $\textcircled{1}$ and $a_0^2 a_2 a_1^{-1}$ from step $\textcircled{2}$ join together to
be $a_1^{-1}a_0a_2 a_1^{-1}$.

\medskip

\noindent
{\bf Step \textcircled{3}}: The next colored digit is $5$. We choose: 
$$s_4 s_3 s_2 s_1 = \left( s_4 s_0^3 \right) \left( s_0^3 s_3 \right) \left( s_2 s_0^3 \right) \left( s_0^3 s_1 \right) = a_4 a_3 a_2 a_1 \in C_5^1,$$
and leave the treatment of the additional generator $s_0$ to the next step (in this case, to the transition between the coloring part and the ordering part, i.e. $a_0^2 \in C_6$). We will elaborate on this point after describing the ordering part.
\end{exa}

\subsubsection{The ordering part}
We turn now to the ordering part. For $1 \leq k \leq n-1$, define the sets $O_k$ as follows:
\begin{eqnarray*}
O_1 & = & \left\{s_{i} s_{i-1}\cdots s_1 s_0^{\frac{r}{2}\varepsilon} \mid 1\leq i\leq n-1,\ \ \varepsilon \equiv i ({\rm mod} \ 2)\right\} \cup \{ 1 \}=\\
& = &\{a_i a_{i-1}\cdots a_1^{1-2\varepsilon} \mid 1 \leq i \leq n-1,\ \
\varepsilon \equiv i ({\rm mod}\ 2) \}\cup \{1\},
\end{eqnarray*}

\noindent
and for $2 \leq k \leq n-1$:
\begin{eqnarray*}
O_k & = & \left\{s_{i} s_{i-1}\cdots s_k s_0^{\frac{r}{2}\varepsilon} \mid k\leq i\leq n-1,\ \ \varepsilon \equiv (i+1-k) ({\rm mod}\ 2)\right\} \cup \{ 1 \}=\\
& = & \{a_ia_{i-1}\cdots a_k \mid k \leq i \leq n-1\} \cup \{1\}.
\end{eqnarray*}

We start by pushing the digit $i_1=|\sigma|(1)$ of $\pi$ to its
correct place. Let $p_1=
|\pi^{-1}|(i_1)$. The pushing will be
done by multiplying $\pi$ (from the right) by the element $o_1 \in
O_1$, where:
$$o_1=\left\{\begin{array}{cc}
s_{p_1-1}s_{p_1-2}\cdots s_1=a_{p_1-1} \cdots a_1 & p_1-1\ {\rm is\ even} \\
s_{p_1-1}s_{p_1-2}\cdots s_1 s_0^{\frac{r}{2}}=a_{p_1-1} \cdots a_2 a_1^{-1} & p_1-1\ {\rm is\ odd}.
\end{array} \right.$$

Now, we continue to push the other digits of $\pi$:  for each $1 < k
\leq n-2$, assuming that the digit $i_k=\sigma(k)$ is now located at position
$p_k$, we use the element $o_k \in O_k$ defined by:
$$o_k=\left\{\begin{array}{cc}
s_{p_k-1}s_{p_k-2}\cdots s_k=a_{p_k-1} \cdots a_k & p_k-k\ {\rm is\ even} \\
s_{p_k-1}s_{p_k-2}\cdots s_k s_0^{\frac{r}{2}}=a_{p_k-1} \cdots a_k & p_k-k\ {\rm is \ odd},
\end{array} \right.$$
in order to push the digit $i_k$ to its correct place in $\pi$. Now, we have two possibilities:
\begin{itemize}
\item The coloring part was completed without remainders, which means that both the coloring part and the ordering part consist of even number of $G_{r,n}$-generators. In this case, we choose $1 \in C_{n+1}$, and we have that:
    $$\pi= \gamma_1 \cdots \gamma_n \cdot 1 \cdot o_{n-1}^{-1} \cdots o_1^{-1},$$
    where $\gamma_i \in C_i$ for $1 \leq i \leq n$, and $o_i \in O_i$ for $1 \leq i \leq n-1$.
\item The coloring part has a remainder of the generator $s_0$. This means that the coloring part required an odd number of $G_{r,n}$-generators, and therefore the ordering part had an odd number of  $G_{r,n}$-generators as well (since the sum of their lengths is even). Thus, in the ordering part, there will be a remainder of $s_0^{\rh}$, so we choose $s_0^{\rh} s_0 = s_0^{\rh +1} = a_0^{\frac{r+2}{4}} \in C_{n+1}$, and we have that:
    $$\pi= \gamma_1 \cdots \gamma_n \cdot a_0^{\frac{r+2}{4}} \cdot o_{n-1}^{-1} \cdots o_1^{-1},$$  where $\gamma_i \in C_i$ for $1 \leq i \leq n$, and $o_i \in O_i$ for $1 \leq i \leq n-1$.
\end{itemize}

In both cases, we have now: $\pi=\sigma \cdot o_{n-1}^{-1} \cdots o_1^{-1}$, and we are done.

\begin{exa}
We continue the previous example. Again, let $\pi=\left( 1 2^{[2]} 4 5^{[1]} 3^{[3]} \right)$. After the completion of the coloring part, we have reached the ordered permutation: $\sigma=\left( 5^{[1]} 3^{[3]} 2^{[2]} 1 4 \right)$.

Now, we go the other way around: we start with $\pi$ and order it to
obtain $\sigma$:
$$\left( 1 2^{[2]} 4 5^{[1]} 3^{[3]} \right) \ \  \stackrel{\textcircled{7} \ s_3s_2 s_1 s_0^3}{\longrightarrow} \ \ \left( 5^{[4]} 1 2^{[2]} 4 3^{[3]} \right)  \stackrel{\textcircled{6} \ s_4 s_3
s_2 s_0^3}{\longrightarrow} \ \ \left( 5^{[1]} 3^{[3]} 1 2^{[2]} 4 \right) \ \  \stackrel{\textcircled{5} \ s_3 s_0^3}{\longrightarrow}$$
$$\rightarrow \ \ \left( 5^{[4]} 3^{[3]} 2^{[2]} 1 4 \right) \ \ \stackrel{\textcircled{4} \ s_0^4}{\longrightarrow} \ \ \left( 5^{[4]} 3^{[3]} 2^{[2]} 1 4 \right) = \sigma.$$

In {\bf step \textcircled{7}}, we push the digit $5$ to its correct place with respect to the ordered permutation $\sigma$. We use $s_3s_2s_1s_0^3=a_3a_2a_1^{-1} \in O_1.$
Note that the digit $5$ is bearing extra three colors.
In {\bf step \textcircled{6}}, we push the digit $3$ into its place, using $s_4s_3s_2s_0^3=a_4a_3a_2 \in O_2$. Note that the color of the digit $5$ is correct again.
Next, in {\bf step \textcircled{5}}, we push $2$ to its correct place using $s_3s_0^3=a_3 \in O_3$.

Now, note that after completing step $\textcircled{5}$, we still did not arrive at $\sigma$, since the digit $5$ has again a wrong color. On the other hand, we have a debt of a generator $s_0$ from the coloring part. Both problems would be solved simultaneously by using $s_0 s_0^3=a_0^2 \in C_{6}$. This is exactly what we have done in {\bf step \textcircled{4}}.
\end{exa}

\medskip


From the above analysis we can conclude that each permutation $\pi \in \arn$ has a {\it canonical decomposition} with respect to the set $\mathcal{A}$. This is the context of the following theorem.

\begin{thm}\label{generate_thm}
The set ${\mathcal A}=\{a_i \mid 0 \leq i \leq n-1 \} \cup \{a_1^{-1}\}$ generates
$A_{r,n}$. Moreover, for each $\pi \in A_{r,n}$, there is a {\em unique
presentation} as:
$$\pi=\gamma_1 \cdots \gamma_n \gamma_{n+1} \cdot o_{n-1}^{-1} \cdots
o_1^{-1},$$
where $\gamma_i \in C_i$ for $1 \leq i \leq n+1$ and $o_j \in O_j$ for $1 \leq j \leq n-1$. This presentation is called {\em the canonical decomposition of $\pi$}.
\end{thm}


\begin{proof}
Let $M$ be the Cartesian product
$$M=C_1 \times \cdots \times C_n \times C_{n+1} \times O_{n-1} \times \cdots \times O_1.$$

We start by defining a subset $L$ of $M$ which we call {\it the set of legal vectors}.  A vector $\vec{\omega} = (\gamma_1, \dots ,\gamma_n, \gamma_{n+1},o_{n-1}, \dots, o_1) \in M$ is called a {\it legal vector} if it satisfies the following two conditions:
\begin{enumerate}
\item Let $i$ and $j$ be two indices satisfying $\gamma_k=1$ for all $i<k<j$, $\gamma_i\neq 1$ and $\gamma_j\neq 1$ (i.e. the digits $i$ and $j$ are colored, but the digits between them are not colored). If $\gamma_i$ ends with $s_0^{r-1}$, then $\gamma_j$ does not start with $s_0$.
\item Let $i$ and $j$ be two indices satisfying $\gamma_k=1$ for all $i<k<j$, $\gamma_i\neq 1$ and $\gamma_j\neq 1$. If $\gamma_i$ ends with $s_1$, then $\gamma_j$ starts with $s_0$.
\end{enumerate}

\medskip

We have to prove the following two claims:
\begin{enumerate}
\item[(a)] The algorithm associates a legal vector in $L$ to any $\pi \in A_{r,n}$. This proves the existence of the presentation.
\item[(b)] $|L|=\frac{r^n n!}{2}(=|A_{r,n}|)$, which implies the uniqueness.
\end{enumerate}

\medskip

Claim (a) is implied immediately from the algorithm, so we pass to the proof of Claim (b). For that, we define the notions of {\it external} and {\it internal} components of a vector:
$$\vec{\omega}=(\gamma_1, \dots, \gamma_n,\gamma_{n+1},o_{n-1}, \dots, o_1) \in L.$$
A component $\gamma_i \in C_i$ is called {\it external}, if it ends either with $s_0^{r-1}$ or with $s_1$ (i.e. a component which imposes a restriction on the next non-trivial component), and {\it internal} otherwise. Note that $\gamma_1$ is always internal, since by definition, the generator $s_1$ does not appear in $\gamma_1$ and it cannot end with the expression $s_0^{r-1}$ since $r-1$ is odd.

For constructing an element of $L$, we start by choosing the element $\gamma_1 \in C_1$, out of $\rh$ possibilities. Next, we choose which components will be external. Note that for each external component there are two possibilities, but each external component restricts the possibilities for the next non-trivial component (i.e. $\gamma_i \neq 1$). Next, for each internal component, we choose one out of $r-1$ possibilities.

When the coloring part is over, we have two possibilities: If we have no remainder from the coloring part, then we choose $\gamma_{n+1}=1$, and we have to complete the process by a permutation of $S_n$ of even length. On the other hand, if we do have a remainder, then we choose $\gamma_{n+1}=a^{\frac{r+2}{4}}$ and we have to complete the process by a permutation of $S_n$ of odd length. Altogether, this contributes $n!$ possibilities for completing the presentation.

\medskip

Following the above discussion, we have that the number of legal vectors is:
$$|L|=\rh \cdot \sumlim_{i=0}^{n-1}{{{n-1} \choose i} (r-1)^{n-1-i}} \cdot n!=\frac{r\cdot n!}{2}((r-1)+1)^{n-1}=\frac{n!}{2}r^n,$$
as needed.
\end{proof}

\begin{rem}
Note that we do not claim that the presentation, described above, is irreducible as it. Take for example the expression for $\pi=1 2^{[2]} 4 5^{[1]} 3^{[3]}$, computed in Example \ref{pairing},
to be: 
$$a_1^{-1} \bold{a_0^2 a_0^2} a_2 a_1^{-1} \bold{a_0^2 a_0^2 a_0^2} a_4 a_3 a_2 a_1 a_0^2 a_3 a_2 a_3 a_4 a_1 a_2 a_3,$$ 
which can be shortened to  $a_1^{-1} a_0 a_2 a_1^{-1}a_4 a_3 a_2 a_1 a_0^2 a_3 a_2 a_3 a_4 a_1 a_2 a_3$. On the other hand, after we cancel all the redundant appearances of $a_0$, we do obtain an irreducible expression, as will be proven in the next section.
\end{rem}

For $\pi \in A_{r,n}$, let $L_A(\pi)$ be the number of generators needed to write $\pi$ as a product of the $A_{r,n}$-generators by the algorithm.

As a consequence of the analysis of the algorithm, we have the following result:

\begin{thm}\label{first con}

\begin{enumerate}
\item Let $\pi \in A_{r,n}$, and let
$\omega=s_0^{z_1}b_1s_0^{z_2}b_2 \cdots s_0^{z_n}b_n$ where $b_i \in \left( {\mathcal S}-\{s_0\} \right)^*$ be its canonical presentation with respect to $\mathcal S$.
Then the translation of $\omega$ to the generators in $\mathcal A$ will be
$$a_0^{z_1 \oslash 2} b_1' a_0^{z_2 \oslash 2} b_2' \cdots a_0^{z_n \oslash 2} b_n',$$
where $b_i' \in \left( {\mathcal A}-\{a_0\} \right)^*$.  

Moreover, if $b_i=s_{i_1} \cdots s_{i_k}$, then
$b_i'=a_{i_1}^{\varepsilon_{i_1}} \cdots a_{i_k}^{\varepsilon_{i_k}}$, where $\varepsilon_{i_j}\in \{\pm 1\}$ if $i_j=1$ and $\varepsilon_{i_j}=1$ otherwise.

\item Let $\pi \in A_{r,n}$. Then:
$$L_A(\pi)=\sumlim_{z_i(\pi) \neq 0}{(i-1)} +{\rm inv}(\pi) + \sumlim_{i=1}^n\left({z_i(\pi)} \oslash 2 \right) .$$
\end{enumerate}

\end{thm}

\begin{proof}
Part (1) is straightforward from the analysis of the algorithm, so we proceed to the proof of part (2).

In \cite[Theorem 4.3]{B}, the algorithm for presenting an element in $\grn$ is described (see also \cite{S}). It is proven that for $\pi \in \grn$, the length of $\pi$, with respect to the generators in $\mathcal S$, is: 
$$\ell_{\grn}(\pi)=\sumlim_{z_i(\pi) \neq 0}{(i-1)} +{\rm inv}(\pi) +\sumlim_{i=1}^n\left({z_i(\pi)}\right),$$ 
so by part (1) we have:
\begin{equation}\label{length wrt grn}
L_A(\pi)=\sumlim_{z_i(\pi) \neq 0}{(i-1)} +{\rm inv}(\pi) + \sumlim_{i=1}^n\left({z_i(\pi)} \oslash 2 \right).
\end{equation}
\end{proof}

Another consequence of the algorithm is the following criterion for an element for being in the group $A_{r,n}$:

\begin{thm}\label{criteria}
Let $\pi \in G_{r,n}$. Then: $\pi \in A_{r,n}$ if and only if:
$${\rm csum}(\pi) + {\rm inv}(|\pi|) \equiv 0 ({\rm mod}\ 2).$$
\end{thm}

\begin{proof}
By the definition, $\pi \in A_{r,n}$ if and only if $\ell_{G_{r,n}}(\pi) \equiv 0 \pmod 2$. By the algorithm described above, $\ell_{G_{r,n}}(\pi)={\rm csum}(\pi)+k$, where $k$ is the number of generators $s_i$, for $i \neq 0$, used in the presentation of $\pi$. On the other hand, if we remove the appearances of $s_0$ from the presentation of $\pi$, we get $|\pi|$. Since lengths of different presentations of the same element of $S_n$ have the same parity, we have that $k \equiv {\rm inv}(|\pi|) ({\rm mod}\ 2)$, and therefore the criterion follows.
\end{proof}

\section{The presentation of $A_{r,n}$ and its length function}\label{proofs}
In \cite{BMR}, Dynkin-like diagrams were presented for the groups $G_{r,n}$. Such diagrams are based on a Coxeter-like presentation.
In this section, we compute a Coxeter-like presentation for $A_{r,n}$,
as well as a Dynkin-like diagram for the groups $A_{r,n}$.

\subsection{The presentation of $A_{r,n}$}

We start with the presentation of $A_{r,n}$.

\begin{thm}\label{thm_pres}
The set ${\mathcal A}=\{a_0,a_1^{\pm 1},\dots,a_{n-1}\}$ generates $A_{r,n}$, subject to the following relations:

\begin{enumerate}
\item $a_0^{\frac{r}{2}}=1$,
\item $a_1^4=1$,
\item $a_i^2=1$ for $i>1$,
\item $a_i a_j=a_j a_i$ for $|i-j|>1$ and $i,j \neq 1$,
\item $(a_ia_{i+1})^3=1$ for $i \geq 1$,
\item $(a_0a_1)^{2r}=1$,
\item $(a_0 a_1^{-1})^{2r}=1$,
\item $a_0 a_1^2 = a_1^2 a_0$,
\item $a_1 a_i =a_i a_1^{-1}$ for $i>2$.
\end{enumerate}

Denote by $\mathcal{R}$ the above set of relations.
\end{thm}

\begin{rem}
Note that relation (5) implies $(a_1^{-1} a_2)^3=1$ too, and
relation (8) implies $a_1^{-1} a_0 a_1^{-1} = a_1 a_0 a_1$ and $a_1^{-1} a_0 a_1 = a_1 a_0 a_1^{-1}$.
\end{rem}

\begin{proof}
We have already shown in Theorem \ref{generate_thm} that $\mathcal A$ generates $A_{r,n}$. Here, we prove that the set $\mathcal{R}$ is the complete set of relations for $\grn$.

We imitate the idea of the proof of Proposition 2.1.1 in \cite{BRR}.

Consider the abstract group $A_{r,n}^+$ generated by the
elements
$${\mathcal A}=\{a_0,a_1,a_1^{-1},\dots, a_{n-1}\},$$
with $\mathcal{R}$ as the set of relations.  Note that the set
mapping $\alpha: {\mathcal A} \rightarrow A_{r,n}^+$ defined by:

$$\alpha(a_1^{\varepsilon})=a_1^{-\varepsilon} \text{ for }  \varepsilon \in \{-1,1\},$$
$$\alpha(a_i)=a_i \text{ for } i \in \{0,2,\dots,n-1\},$$
extends to a group automorphism $\alpha$ on $A_{r,n}^+$.
Indeed, considering $A_{r,n}$ as a subgroup of $G_{r,n}$, $\alpha$ is the inner automorphism
defined by the conjugation by $s_0^{\rh}$.

Thus, the group $\mathbb{Z}_2=\{1,\alpha\}$ acts on $A_{r,n}^+$ and we have the semidirect product $A_{r,n}^+ \rtimes \mathbb{Z}_2$, where the product is defined as follows:
$$(x_1 \alpha^i)\cdot (x_2 \alpha^j)=x_1\alpha^i(x_2) \cdot \alpha^{i+j}.$$
The semidirect product has the following presentation:
\begin{equation}
A_{r,n}^+ \rtimes \mathbb{Z}_2 = \langle \alpha,a_0,a_1,a_1^{-1},a_2,\dots,a_{n-1}\mid \bold{R},\alpha
a_i \alpha=\alpha(a_i) \mbox{ for all } i \rangle. \label{presentation of semidirect}
\end{equation}

We prove now that $G_{r,n} \cong A_{r,n}^+ \rtimes \mathbb{Z}_2$. In
order to do this, we define the following two homomorphisms, which are inverses of each other:
$$\rho : A_{r,n}^+ \rtimes \mathbb{Z}_2 \rightarrow G_{r,n},$$
defined on the generators by:
\begin{eqnarray*}
\alpha & \mapsto & s_0^{\rh} \\
a_0 & \mapsto & s_0^{2} \\
a_i & \mapsto & s_0^{\rh}s_i \mbox{ for } i \geq 1, \\
\end{eqnarray*}
and
$$\varphi :G_{r,n} \rightarrow A_{r,n}^+ \rtimes \mathbb{Z}_2,$$
defined on the generators by:
\begin{eqnarray*}
s_0 & \mapsto & \alpha a_0^{\frac{r+2}{4}} \\
s_i & \mapsto & \alpha a_i \mbox{ for } i \geq 1.\\
\end{eqnarray*}

It is easy to see that $\rho$ and $\varphi$ are isomorphisms, so we have that $G_{r,n} \cong A_{r,n}^+ \rtimes \mathbb{Z}_2$.

Now, since $\rho(\arn^+) \subseteq \arn$ and $\arn,\arn^+$ are both subgroups of $\grn$ of index $2$, they must be isomorphic.
\end{proof}

The relations defining $A_{r,n}$ can be graphically described by the following Dynkin-like diagram, where the numbers inside the circles are the orders of the corresponding generators, an edge without a label between two circles means that the order of the multiplication of the two corresponding generators is $3$, and an edge labeled $2r$ between two circles means that the order of the multiplication of the two corresponding generators is $2r$ (two circles with no connecting edge mean that the two corresponding generators commute):

\begin{figure}[!ht]
\epsfysize=3cm

\epsfbox{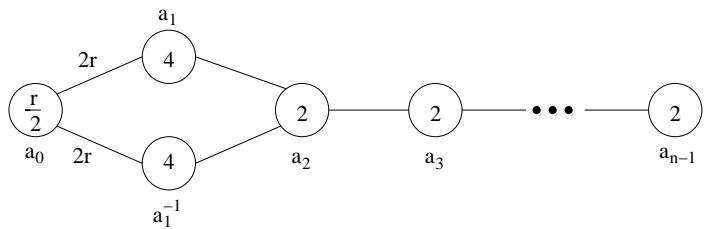}
\caption{Dynkin-like diagram of $A_{r,n}$} \label{Dynkin-like}
\end{figure}

\subsection{The length function}

Given a group $G$, generated by a set $A$, we denote by $\ell_A$ the length function on $G$ with respect to $A$. Explicitly, for each $\pi \in G$:
$$\ell_A(\pi)=\min\{u \mid g=a_1 \cdots a_u, \mbox{ where } a_i \in A\}.$$
In this section, we prove that the algorithm described above, indeed gives us a reduced word with respect to the set of generators $\mathcal A$. In other words, we prove that for each $\pi \in A_{r,n}$,
$\ell_{\mathcal A}(\pi)=L_A(\pi)$, where $L_A(\pi)$ is the number of generators in the presentation of $\pi$, obtained by the algorithm (and was computed in Theorem \ref{first con}(2)).

\medskip

We start with the following set of definitions:

\begin{defn}
\begin{itemize}
\item Let $\omega=b_0s_0^{i_1}b_1 s_0^{i_2}b_2 \cdots s_0^{i_k}b_ks_0^{i_{k+1}} b_{k+1}$ be {\it any} factorization of $\pi \in A_{r,n}$ as a product of the generators of $\mathcal S$, such that $b_i \in \left( {\mathcal S}-\{s_0\} \right)^*$ for $0 \leq i \leq k+1$.

\item Let  $n(\omega)$ be the number of generators $s_i$ in $\omega$ such that $i \neq 0$.

\item Let $\beta(\omega)=\sumlim_{u=1}^{k+1}({i_u \oslash 2})$.

\item Let $\mu(\pi)$ be the minimal value of $\beta(\omega)+n(\omega)$, obtained over all possible factorizations, $\omega$, of $\pi$ to $G_{r,n}$-generators as above.
\end{itemize}
\end{defn}

\begin{thm}
Let $\pi \in A_{r,n}$. Then: $\ell_{\mathcal A}(\pi)=\mu(\pi)$.
\end{thm}

\begin{proof}
Let $\pi \in A_{r,n}$. We start by proving the inequality: $\ell_{\mathcal A}(\pi) \geq \mu(\pi)$. Let $\omega$ be a reduced word in generators from $\mathcal A$.
Apply the map $\rho$, defined in the proof of Theorem \ref{thm_pres} above, which sends $a_i$ to $s_0^{\rh}s_i$ for $i \in \{1,\dots,n-1\}$, $a_1^{-1}$ to $s_1s_0^{\rh}$, and $a_0$ to $s_0^2$ on each letter separately, and concatenate.

This yields a word $\eta$, factorizing $\pi$ in $G_{r,n}$-generators, satisfying: $\ell_{\mathcal A}(\pi)=\beta(\eta)+n(\eta)$. This implies that the minimal value of $\mu$ over {\it all} the factorizations of $\pi$ to $G_{r,n}$-generators is at most $\ell_{\mathcal A}(\pi)$, since $\mu (\pi)$ is defined as the minimal value of all such expressions. So, we have: $\ell_{\mathcal A}(\pi) \geq \mu(\pi)$ (see Example \ref{exam_sec5}(a) below).

\medskip

In order to prove the opposite inequality: $\mu(\pi) \geq \ell_A(\pi)$, let $\omega$ be a factorization of $\pi$ to $G_{r,n}$-generators which achieves the minimal value of $\mu(\pi)$. Apply the map $\varphi$, defined in the proof of Theorem \ref{thm_pres}, which sends $s_0$ to $\alpha  a_0^{\frac{r+2}{4}}$, and for all $i \geq 1$ sends $s_i$ to $\alpha a_i$ on each letter separately, and
concatenate.  This gives us a  factorization $\eta$ of $\pi$ in ${\mathcal A} \cup \{\alpha\}$-generators, having $n(\omega)$ generators $a_i$ with $i \neq 0$ and $\beta(\omega)$ occurrences of $a_0$. The factorization $\eta$ contains also an even number of occurrences of the letter $\alpha$ (one occurrence for each $s_i$ for $i \geq 0$; recall that the number of $s_i$'s is even by definition).  By using the relations $\alpha a_i \alpha=a_i$ for $1 \leq i \leq n$ and $\alpha a_1^{-1} \alpha=a_1^{-1}$, we can cancel out all occurrences of $\alpha$ and we have an ${\mathcal A}^*$-word (i.e. a word written using generators from $\mathcal A$) factorization of $\pi$ of length $\mu(\pi)$ (see Example \ref{exam_sec5}(b) below). This proves that  $\mu(\pi) \geq \ell_{\mathcal A}(\pi)$.
\end{proof}

\begin{exa}\label{exam_sec5} \ \\
(a) We illustrate the proof of the first direction of the above proof: Let  $\pi=\left( 1^{[5]} 3^{[3]} 2^{[3]} 4^{[0]} \right) \in A_{6,4}$. The $\mathcal{A}^*$-word  $w=a_0a_1^{-1}a_2 a_1a_2a_1^{-1}$ is a reduced factorization of $\pi$. Apply the map $\rho$ on $\omega$ to get: $$\eta=s_0^2s_1s_0^3s_0^3s_2s_0^3s_1s_0^3s_2s_1s_0^3.$$ 
One can calculate that:
\begin{eqnarray*}
\beta(\eta)+n(\eta) &=&(2 \oslash 2)+(6 \oslash 2)+ (3 \oslash 2)+ (3 \oslash 2)+(3 \oslash 2)+5 = \\
& = & 1+0+0+0+0+5=6=\ell_{\mathcal A}(\pi).\\
\end{eqnarray*}

\medskip

\noindent
(b) In this part, we illustrate the derivation described in the second direction of the above proof: Let $\pi=\left( 1^{[2]} 2^{[0]} 4^{[0]} 3^{[1]} \right)$, and let $\omega=s_0^2 s_2 s_1 s_0 s_1 s_2 s_3$ be a factorization of $\pi$ (here $\beta(\omega)+n(\omega)$ is indeed minimal). Then:
$$\varphi(\omega)=\alpha a_0^2 \alpha a_0^2 \alpha a_2 \alpha a_1 \alpha a_0^2 \alpha a_1 \alpha a_2 \alpha a_3=\eta.$$
Now, after canceling out the appearances of $\alpha$, we get the ${\mathcal A}^*$-word:
$$a_0^2 a_0^2 a_2 a_1^{-1} a_0^2  a_1^{-1}  a_2  a_3=a_0 a_2 a_1^{-1} a_0^2  a_1^{-1}  a_2  a_3.$$
\end{exa}

\begin{thm}
Let $\pi \in A_{r,n}$. Then:
$$L_A(\pi)=\mu(\pi).$$
\end{thm}

\begin{proof}
Let $\pi \in A_{r,n}$. The canonical decomposition of $\pi$ demonstrates the inequality  $L_A(\pi) \geq \mu(\pi)$. For the opposite inequality, assume to the contrary that there is some decomposition $\omega$ of $\pi$ satisfying: 
$$n(\omega)+\beta(\omega)< L_A(\pi).$$ 
This contradicts the fact that Equation (\ref{length wrt grn}) is the length of $\pi$ with respect to the usual generators of $\grn$.
\end{proof}

Hence, we have obtained the following corollary, which is the main result of this section:
\begin{cor}
The function $L_A$ is indeed the length function of the group of colored alternating permutations with respect to the set of generators ${\mathcal A}$. Explicitly: for each $\pi \in A_{r,n}$,
$$\ell_{\mathcal A}(\pi)=L_A (\pi).$$
\end{cor}

Consequently, we can easily get the generating function for the length function:
\begin{cor}\label{gen fun}.
The generating function for the length function $\ell_{\mathcal A}$ is:
$$\sumlim_{\pi \in A_{r,n}}{q^{\ell_{\mathcal A}(\pi)}} = \frac{1}{2}[n]!_q\prod\limits_{j=1}^{n}\left({1+q^{j-1}(1+2q+\cdots+2q^{\frac{r}{2}-1})}\right).$$
\end{cor}

\begin{proof}
Each $\pi \in A_{r,n}$ has a canonical decomposition
into generators from the set $\mathcal S$. By Theorem 4.4 of \cite{B}, the generating function for the
length with respect to $\mathcal S$ over the whole group $\grn$ is:
$$\sumlim_{\pi \in G_{r,n}}{q^{\ell_{\mathcal S}(\pi)}} = [n]!_q\prod\limits_{j=1}^{n}\left({1+q^{j-1}(1+q+\cdots+q^{r-1})}\right).$$
By Theorem \ref{first con}(1), the factor $s_{j-1} \cdots s_1 s_0^z$ which colors the digit $j$ by $z$ colors is converted to $a_{j-1} \cdots a_1^{\pm 1} a_0^{z \oslash 2}$. Since the mapping $\Z_r \to \Z_{\rh}$ is a 2:1-epimorphism, the factor $\left({1+q^{j-1}(1+q+\cdots+q^{r-1})}\right)$ is converted to $\left({1+q^{j-1}(1+2q+\cdots+2q^{\frac{r}{2}-1})}\right)$.
Finally, after completing the coloring part, only half of the permutations of $S_n$ are permitted (since the total length of the word in $\grn$-generators should be even), so we have to divide the generating function by $2$.
\end{proof}

\section{The colored alternating group as a covering group}\label{cover}
In \cite{RR}, a covering map $f:A_{n+1} \rightarrow S_n$ was defined and used to lift some identities of $S_n$ to $A_{n+1}$. In this section, we use a similar technique with a covering map from $A_{r,n}$ to $G_{\rh,n}$. Unlike the case of $A_{n+1}$, this map is an epimorphism, and hence the kernel of this map will be combinatorially described. We also present a section $s:G_{\rh,n} \rightarrow \arn$ which gives us a way to decompose the length function, $\ell_{\mathcal A}(\pi)$, into two summands, one of them is constant on the coset of $\pi$, while the other, which will be called the {\it fibral length}, varies over the coset.
We present a nice combinatorial interpretation of the last parameter, as well as a generating function for it over each coset.  In Section \ref{perm stat}, we use this covering map to lift some identities and permutation statistics from $G_{\rh,n}$ to $A_{r,n}$.

\medskip

Define the following projection:
$$p:A_{r,n} \rightarrow G_{\frac{r}{2},n},$$
as follows: if $\pi=\left(b_1^{[c_1]} \cdots b_n^{[c_n]}\right)$, then:
$$p(\pi)=\left(b_1^{[c_1 \oslash 2]} \cdots b_n^{[c_n \oslash 2]}\right).$$

\begin{exa}
Let $\pi=\left( 3^{[0]} 2^{[1]} 4^{[2]} 1^{[3]} \right) \in G_{6,4}$. Then: $$p(\pi)=\left( 3^{[0 \oslash 2]} 2^{[1 \oslash 2]} 4^{[2 \oslash 2]} 1^{[3 \oslash 2]} \right)= \left( 3^{[0]} 2^{[2]} 4^{[1]} 1^{[0]} \right).$$
\end{exa}

Then, we have:
\begin{lem}
The map $p$ is an epimorphism. Moreover, the kernel of $p$ is the
normal closure of $a_1^2$ in $A_{r,n}$.
Thus: $G_{\rh,n} \cong \frac{A_{r,n}} {\ll a_1^2 \gg}$.
\end{lem}

\begin{proof}
The map $p$ is clearly a homomorphism since the operator $\oslash$ commutes with the addition operation in
$\mathbb{Z}_{\frac{r}{2}}$ (see Equation
(\ref{oslash commutes})). Now, if $\sigma=\left(b_1^{[c_1]} \cdots b_n^{[c_n]}\right)
\in G_{\rh,n}$, and $j \in \{1,\dots,n\}$ satisfies $b_j=1$, then, by Theorem \ref{criteria},  we have either
$$\left( b_1^{[2c_1]}
\cdots b_j^{[2c_j]} \cdots b_n^{[2c_n]} \right) \in A_{r,n} \mbox{ or } \left( b_1^{[2c_1]} \cdots b_j^{\left[ 2c_j+\rh \right]} \cdots b_n^{[2c_n]} \right) \in A_{r,n},$$ where the computations are made modulo $r$. This implies that $p$ is an epimorphism.

It remains to find the kernel. Since $a_1^2=\left(1^{\left[ \rh \right]} 2^{ \left[ \rh \right]} 3^{[0]}
\cdots n^{[0]}\right)$, we have $a_1^{2} \in \ker (p)$, thus $\ll a_1^2 \gg \leq \ker(p)$. On the other hand, all the elements of $\ker (p)$ are of the form $\left( 1^{[c_1]} 2^{[c_2]} \cdots n^{[c_n]} \right)$, where $c_i \in \left\{ 0,\rh \right\}$ and $\{i \mid c_i \neq 0\}|$ is even.
For each $i<j$, we can use the element
$t_{i,j}a_1^2 t_{i,j}^{-1} \in \ll a_1^2 \gg$, where $t_{i,j}=s_{i-1}s_{i-2}\cdots s_1 \cdot s_j s_{j-1} \cdots s_2$ in order to color digits $i$ and $j$ in $\rh$ colors without touching the other digits.  This proves that $\ll a_1^2 \gg =\ker (p)$, as needed.
\end{proof}

We emphasize the following two observations, which can be concluded from the proof of the previous theorem,  for a future use.

\begin{obs}\label{inside the fiber by r/2}
\begin{enumerate}
\item $|\ker (p)|=2^{n-1}$.
\item Let $\pi, \pi' \in A_{r,n}$ be such that $p(\pi)=p(\pi')$. Then, for each $i \in \{1,\dots ,n\}$, $c_i(\pi) \equiv c_i(\pi') \pmod \rh$.
Moreover, $c_i(\pi)$ and $c_i(\pi')$ differ by $\rh$ for an {\em even} number of indices.
\end{enumerate}
\end{obs}

The following obvious lemma presents the action of $p$ on the generators of $A_{r,n}$:

\begin{lem}\label{algebraic form of p}
$$\begin{array}{l}
  p(a_i)=s_i \mbox{ for } 1 \leq i \leq n-1, \\
  p(a_1^{-1})=s_1, \\
  p(a_0)=s_0. \end{array}$$
\qed
\end{lem}

We introduce the following section of the covering map $p$:
Define $$s: G_{\rh,n} \rightarrow A_{r,n}$$
as follows:
if $\pi=\left(p_1^{[c_1]} \cdots p_n^{[c_n]}\right)$ and $p_j=1$, then:
$$ \pi_0=s(\pi)=\left\{\begin{array}{cc}
                   \left(p_1^{[2c_1]} \cdots p_j^{[2c_j]} \cdots p_n^{[2c_n]}\right) & {\rm inv}(|\pi|)\equiv 0 ({\rm mod}\ 2)  \\
                   \left(p_1^{[2c_1]} \cdots p_j^{\left[ 2c_j+\rh \right]} \cdots p_n^{[2c_n]}\right) & {\rm inv}(|\pi|) \equiv 1 ({\rm mod}\ 2),
                   \end{array}\right.$$
where the computations are made modulo $r$. It is easy to verify that $p \circ s={\rm Id}$.

\begin{exa}
Let $\pi =\left(2^{[1]} 3^{[0]} 4^{[1]} 1^{[2]}\right) \in G_{3,4}$. Then, ${\rm inv}(|\pi|)=3$ and $j=4$. Thus,
$\left(2^{[2]} 3^{[0]} 4^{[2]} 1^{[4]}\right) \notin A_{6,3}$, so $s(\pi)=\left(2^{[2]} 3^{[0]} 4^{[2]} 1^{[1]}\right) \in A_{6,3}$.
\end{exa}

%
%
%

\subsection{The fibral length}
For each $\pi \in A_{r,n}$, the length function of $\pi$ with respect to the set of generators $\mathcal A$ can be decomposed into two summands. The first summand is the length of $p(\pi)$ as an element in $G_{\frac{r}{2},n}$, which is obviously invariant on the fiber of $\pi$.  The second summand, which varies along the fiber, will be called the {\it fibral
length}. As will be shown in this section, it has a nice combinatorial interpretation. We start with the definition of the {\it fibral length}.

\begin{defn}\label{fibral_def}
For each $\pi \in
A_{r,n}$, define the {\it fibral length} of $\pi$ to be:
$$\ell_F(\pi)=\ell_{\mathcal A}(\pi)-\ell_{\mathcal A}(s(p(\pi))).$$
\end{defn}

For $\pi \in G_{r,n}$, denote
$c(\pi)=\sumlim_{z_i(\pi)\neq 0}{(i-1)}.$ By the definition of\break $\pi_0=s(p(\pi))$, we have:
\begin{equation}\label{c of pi minus c of pi0}
c(\pi)-c(\pi_0)=\sumlim_{\left\{i \mid z_i(\pi) = \rh\right\}}{(i-1)}.
\end{equation}

We will need the following two lemmata in the sequel:

\begin{lem}\label{fibral length as a difference}
$$\ell_F(\pi)=c(\pi)-c(\pi_0)+{\rm inv}(\pi)-{\rm inv}(\pi_0).$$
\end{lem}

\begin{proof}
Let $\pi \in A_{r,n}$. Then, by Observation \ref{inside the fiber by r/2}(2),
$$\sumlim_{i=1}^n\left({z_i(\pi) \oslash 2}  \right)=\sumlim_{i=1}^n({z_i(\pi_0) \oslash 2} ).$$
By Theorem \ref{first con}(2), we are done.
\end{proof}

\begin{lem}\label{l_F positive}
For each $\pi \in A_{r,n}$, we have: $\ell_F(\pi) \geq 0$.
\end{lem}

\begin{proof}
Let $\pi \in A_{r,n}$. By Lemma \ref{fibral length as a difference}, we have: $$\ell_F(\pi)=c(\pi)-c(\pi_0)+{\rm inv}(\pi)-{\rm inv}(\pi_0).$$
By Equation (\ref{c of pi minus c of pi0}), $c(\pi) -c(\pi_0) \geq 0$.
Now, let $1 \leq k<m \leq n$ be such that $\pi(k)=i^{[\alpha]}<j^{[\beta]}=\pi(m)$, but $\pi_0(k)=i^{[\alpha ']} > j^{[\beta']}=\pi_0(m)$. A rather tedious but easy calculation should convince the reader that the only possibility is $|\pi(k)|=i>j=|\pi(m)|$ with $\alpha=\rh$ and $\beta=0$ (and hence $\alpha'=\beta'=0$), so that the digit $i$ is colored in $\pi$ but not in $\pi_0$. Now, since this situation can occur at most $i-1$ times, the contribution of  $i$ to $c(\pi)-c(\pi_0)$ which is $i-1$, will cancel the corresponding negative contribution to ${\rm inv}(\pi)-{\rm inv}(\pi_0)$, and hence the total sum will be positive.
\end{proof}

In order to make the previous proof a bit more accessible, we provide an example.

\begin{exa}
Let $\pi=\left(3^{[3]} 2 1\right) \in A_{3,4}$, so that $\pi_0=(3 2 1)$. Then $i=3$ contributes $2$ to $c(\pi)-c(\pi_0)$ and $-2$ to
${\rm inv}(\pi)-{\rm inv}(\pi_0)$ since $3^{[3]} <2$, $3^{[3]}<1$, but $3>1,3>2$. So, the total sum remains positive.
\end{exa}

Denote by $\ell_G$ the length function of the group $G_{\rh,n}$. Then, we have:
\begin{lem}\label{l_A=l_G}
Let $\pi \in A_{r,n}$ and $\pi_0=s(p(\pi))$. Then: 
$$\ell_{\mathcal A}(\pi_0) = \ell_G(p(\pi)).$$
\end{lem}

\begin{proof}
Let $\pi \in A_{r,n}$. Then:
\begin{eqnarray*}
(*) \qquad \ell_G(p(\pi)) =  \sumlim_{z_i(p(\pi)) \neq 0}{(i-1)}+{\rm inv}(p(\pi))+\sumlim_{i=1}^{n}{z_i (p(\pi))}=\\
=\sumlim_{z_i(\pi) \notin \left\{0,\rh\right\}}{(i-1)}+{\rm inv}(p(\pi))+\sumlim_{i=1}^{n}{(z_i(\pi) \oslash 2)}.
\end{eqnarray*}

Now, $\sumlim_{z_i(\pi) \notin \left\{0,\rh\right\}}{(i-1)} = \sumlim_{z_i(\pi_0) \neq 0}{(i-1)}$, since for all $1<i \leq n$, $z_i(\pi_0) \neq \rh$.
Moreover, it is easy to check case-by-case that ${\rm inv}(p(\pi))={\rm inv}(\pi_0)$.

Finally, it is easy to check that for each $1 \leq i \leq n$, one has:\break $z_i(\pi_0) \oslash 2= z_i(\pi) \oslash 2$
and thus $\sumlim_{i=1}^{n}{(z_i(\pi) \oslash 2)}=\sumlim_{i=1}^{n}{(z_i(\pi_0) \oslash 2)}$.

Therefore, we have:
\begin{eqnarray*}
\ell_G(p(\pi)) & \stackrel{(*)}{=} & \sumlim_{z_i(\pi) \notin \left\{0,\rh\right\}}{(i-1)}+{\rm inv}(p(\pi)) + \sumlim_{i=1}^{n}{(z_i(\pi) \oslash 2)} = \\
& = & \sumlim_{z_i(\pi_0) \neq 0}{(i-1)}+{\rm inv}(p(\pi_0))+\sumlim_{i=1}^{n}{(z_i(\pi_0) \oslash 2)}=\ell_{\mathcal A}(\pi_0),
\end{eqnarray*}
as required.
\end{proof}

As a corollary, we now have by Definition \ref{fibral_def} and Lemma \ref{l_A=l_G}:
\begin{cor}
Let $\pi \in A_{r,n}$ and let $p(\pi)$ be its projection into $G_{\rh,n}$. Then:
$$\ell_{\mathcal A}(\pi)= \ell_F(\pi)+\ell_G(p(\pi)).$$
\qed
\end{cor}

\subsection{A combinatorial interpretation of the fibral length}
For presenting the fibral length in a combinatorial way, we introduce the following parameter on $\arn$.

\begin{defn}
For each $\pi \in A_{r,n}$, define the set of {\it absolute transparent inversions} by:
$${\rm Tinv}(\pi) =\left\{(i,j)\mid z_i(\pi)=\rh,\ i>j, \text{ and } |\pi|^{-1}(i)<|\pi|^{-1}(j)\right\}.$$

Define also:
$${\rm tinv}(\pi)=|{\rm Tinv}(\pi)|.$$
\end{defn}

\begin{exa}
If $\pi=\left(2^{[2]} 4^{[4]} 3^{[3]} 1^{[5]}\right) \in A_{6,4}$, then
${\rm Tinv}(\pi)=\{(3,1)\}$, since $\left( 3^{[3]}, 1^{[5]} \right)$ is an absolute transparent inversion. Hence, ${\rm tinv}(\pi)=1$.
\end{exa}

We have now:
\begin{thm}\label{combinatorial lf}
Let $\pi \in A_{r,n}$. Then:
\begin{equation}\label{combinatorial lf formula}
\ell_F(\pi)=2 \cdot \sum_{\left\{i \mid z_i(\pi) =
\rh\right\}}{(i-1)}-2{\rm tinv}(\pi).
\end{equation}

\end{thm}

\begin{proof}
By Equation (\ref{c of pi minus c of pi0}) and Lemma \ref{fibral length as a difference}, it is sufficient to show that:
\begin{equation}\label{prove}
c(\pi)-c(\pi_0)={\rm inv}(\pi)-{\rm inv}(\pi_0)+2{\rm tinv}(\pi).
\end{equation}

Let $1 < i \leq n$ be such that $z_i(\pi)=\rh$. Then, by Equation (\ref{c of pi minus c of pi0}), the contribution of $i$ to the left hand side is $i-1$, so we have to show that $i$ contributes the same to the right hand side, i.e. for each $1 \leq j<i$, the pair $(i,j)$ contributes $1$ to the expression ${\rm inv}(\pi)-{\rm inv}(\pi_0)+2{\rm tinv}(\pi)$. This can be easily done by a subtle, though, direct check. Note that $i=1$ contributes $0$ to both sides.

If $1 \leq i \leq n$ satisfies $z_i(\pi) \neq \rh$, then $i$ contributes $0$ to both sides.
\end{proof}

We are interested in the distribution of the fibral length of $\pi \in A_{r,n}$ over the coset containing $\pi$. Define:
$$F(\pi)=\sumlim_{\sigma \in p^{-1}(p(\pi))}q^{\ell_F(\sigma)}.$$
It would be much easier to calculate this distribution if we would have translated Theorem \ref{combinatorial lf} to the language of Lehmer codes \cite{lehmer}.
\label{Lehmer code}
Recall that the {\it Lehmer code} of a permutation $\pi \in S_n$ is defined by:
$$L(\pi)=\left( l_{\pi(1)} \cdots l_{\pi(n)} \right),$$
where for each $1 \leq i \leq n$,
$$l_i=|\{j \mid j>i, \pi^{-1}(i)>\pi^{-1}(j)\}|.$$
For example, if $\pi=(31452) \in S_5$ (in window notation), then: $$L(\pi)=(l_3l_1l_4l_5l_2)=(20110).$$
Note that this definition is slightly different from the usual definition of the Lehmer code.

For $\pi \in A_{r,n}$, let $L(|\pi|)=(l_1 \cdots \l_n)$, and define:
$$ \varepsilon_i(\pi)=\left\{\begin{array}{cc}
                   1 & z_i(\pi)=\rh  \\
                   0  & z_i(\pi) \neq \rh.
                   \end{array}\right.$$


The parameter ${\rm tinv}(\pi)$ can be written as:
$${\rm tinv}(\pi)=\sumlim_{i=1}^n{l_i \varepsilon_i(\pi)},$$
and Equation (\ref{combinatorial lf formula}) can be restated as:
\begin{equation}\label{Lehmer form of tinv}
\ell_F(\pi)=2\sumlim_{i=1}^n{\varepsilon_i(\pi)((i-1)- l_i}).
\end{equation}

The following theorem presents an expression for $F(\pi)$, using Equation (\ref{Lehmer form of tinv}).

\begin{thm}
Let $\pi \in A_{r,n}$ and define for each $2 \leq i \leq n$:
$$ \delta_i(\pi)=\left\{\begin{array}{cc}
                   1 & z_i(\pi)\in \left\{0,\rh\right\}  \\
                   0  & z_i(\pi) \notin \left\{0,\rh\right\}.
                   \end{array}\right.$$
Then:
$$F(\pi)=\prod\limits_{i=2}^{n}{\left(1+q^{2\delta_i(i-1-l_i)}\right)}.$$
\end{thm}

\begin{proof}
Let $\pi=\left(a_1^{[c_1]} \cdots a_n^{[c_n]}\right) \in \arn$. By Observation \ref{inside the fiber by r/2}, the coset of
$\pi$ is $\left\{ \left( a_1^{[d_1]} \cdots a_n^{[d_n]} \right) \right\}$, where the vector $(d_1,\dots,d_n)$ is obtained from the vector $(c_1,\dots ,c_n)$ by an addition of $\rh$ to an even number of coordinates. Hence, the vectors $(d_1,\dots,d_n)$, appearing as colors of elements of the coset of $\pi$ can be seen as forming the $(n-1)-$ dimensional affine subspace  $(c_1,\dots,c_n)+Sp(\{e_1-e_2,\dots ,e_{n-1}-e_n\})$,
where $e_i=(0,\dots,\rh,\dots,0)$ (i.e. $\rh$ in the $i$-th coordinate and $0$ elsewhere). Finally, note that when we run over all the elements of the coset, only the coordinates with $z_i(\pi) \in \left\{ 0,\rh \right\}$ and $i \neq 1$ contribute to $\ell_F(\pi)$.
\end{proof}

\section{Some permutation statistics}\label{perm stat}
In this section, we present some permutation statistics for the group of alternating colored permutations.

\subsection{Passing parameters from $G_{\rh,n}$ to $A_{r,n}$}
We exhibit now how to pass parameters defined on the full group of colored permutations of half the number of colors to the group of alternating colored permutations. In order to do that, we have to define the notion of a {\it fiber-fixed parameter}.

\bde
Let $f_A:A_{r,n} \rightarrow \mathbb{N}$ and $f_K:G_{\rh,n} \rightarrow \mathbb{N}$ be two permutation statistics. The parameter $f_A$ is called {\it fiber-fixed} if for each $\pi \in A_{r,n}$,
\begin{equation}\label{fiber-fixed}
f_K(p(\pi))=f_A(\pi).
\end{equation}
\ede

By Observation \ref{inside the fiber by r/2}(1), we have the following connection between the corresponding generating functions:

\begin{lem}
Let $f_A:A_{r,n} \rightarrow \mathbb{N}$ and $f_K:G_{\rh,n} \rightarrow \mathbb{N}$ be such that
$f_A$ is a fiber-fixed parameter. Then:
$$\sumlim_{\pi \in A_{r,n}}{q^{f_A(\pi)}}=2^{n-1}\sumlim_{\pi \in G_{\rh,n}}{q^{f_G(p(\pi))}}.$$
\end{lem}

\begin{proof}
$$\sumlim_{\pi \in A_{r,n}}{q^{f_A(\pi)}}=\sumlim_{\pi \in G_{\rh,n}}{\sumlim_{v \in p^{-1}(\pi)}{q^{f_A(v)}}}
=\sumlim_{\pi \in G_{\rh,n}}{\sumlim_{v \in p^{-1}(\pi)}{q^{f_K(p(v))}}}=$$
$$=2^{n-1}\sumlim_{\pi \in G_{\rh,n}}{q^{f_K(\pi)}}.
$$
\end{proof}

We give two examples offiber-fixed parameters, the first one is the flag-inversion number, and the second is the right-to-left minimum.

\subsection{The flag-inversion number}

The {\it flag-inversion number} was introduced by Foata and Han \cite{FH1,FH2}. Adin, Brenti and Roichman \cite{ABR} used it as a rank function for a weak order on the groups $\grn$. We introduce it in $\grn$:

\begin{defn}
Let $\pi \in \grn$. The {\it flag-inversion number} of $\pi$ is defined as:
$${\rm finv}(\pi)=r \cdot {\rm inv}(|\pi|)+{\rm csum}(\pi).$$
\end{defn}

In \cite{F}, the generating function of ${\rm finv}$ over $\grn$ was computed:
\begin{prop}
$$\sumlim_{\pi \in \grn}{q^{{\rm finv}(\pi)}}=\prodlim_{i=1}^n{[ri]_q}.$$
\end{prop}

We define here a version of the flag-inversion number for the alternating colored permutations whose generating function over $A_{r,n}$ can be computed using on the covering map $p$.

\begin{defn}
Let $\pi \in A_{r,n}$. Define:
$${\rm finv}_A(\pi)=\rh \cdot {\rm inv}(|\pi|)+\sumlim_{i=1}^n{(c_i(\pi) \oslash 2)}.$$
\end{defn}

It is easy to see that the parameter ${\rm finv}$ is indeed fiber-fixed, in the sense of Equation (\ref{fiber-fixed}). Explicitly, for each
$\pi \in A_{r,n}$, ${\rm finv}_A(\pi)={\rm finv}(p(\pi))$. Consequently, we have:

\begin{thm}\label{finv}
$$\sumlim_{\pi \in A_{r,n}}{q^{{\rm finv}_A(\pi)}}=2^{n-1}\prodlim_{i=1}^n{\left[\rh \cdot i\right]_q}. $$
\end{thm}

\subsection{The right-to-left minima}

Another parameter, which can be computed using the covering map, is the {\it right-to-left minima}.

\begin{defn}
Let $p=(a_1,\dots ,a_n)$ be a word over an ordered alphabet $(\Sigma,<)$. Then $a_i \in \{1,\dots ,n\}$ is a {\it right-to-left minimum} if for any $j>i$, one has: $a_j>a_i$. The number of right-to-left minima will be denoted by ${\rm RtlMin}(p)$.
\end{defn}

\begin{exa}
Let $p=(31254) \in S_5$ (with the natural order), then $1,2,4$ are right-to-left minima and hence: ${\rm RtlMin}(p)=3$.
\end{exa}

Regev and Roichman \cite{RR2} defined a version of the right-to-left minima for $\grn$ as follows:

\begin{defn}
Let $\pi=\left(a_1^{[c_1]} \cdots a_n^{[c_n]}\right) \in G_{r,n}$. Define: $${\rm RtlMin}(\pi)=|\{a_i \mid \forall j>i: a_j>a_i,c_i \neq 0\}|.$$
\end{defn}

They showed that the distribution of the parameter ${\rm RtlMin}$ over the full group of colored permutations is (see Proposition 5.1 in \cite{RR2} with $L=\{0,\dots,r-1\}$):
$$\sumlim_{\pi \in \grn}{q^{{\rm RtlMin}(\pi)}}=((r-1)q+1)((r-1)q+r+1)\cdots ((r-1)q+(n-1)r+1).$$

We introduce here a version of the right-to-left minima for $A_{r,n}$.
\begin{defn}
Let $\pi \in A_{r,n}$. Define:
$${\rm RtlMin}_A(\pi)=\left|\left\{a_i \left| \forall j>i,\ a_j>a_i,\ c_i \neq \left\{0,\rh\right\}\right.\right\}\right|.$$
\end{defn}

Again, it is easy to see that ${\rm RtlMin}_A$ is fiber-fixed. Explicitly, for each $\pi \in A_{r,n}$, we have
${\rm RtlMin}_A(\pi)={\rm RtlMin}(p(\pi))$.
Hence, as an immediate corollary of Proposition 5.1 of \cite{RR2}, we get:

\begin{thm}\label{rtl}
$$\sum\limits_{\pi \in A_{r,n}}{q^{{\rm RtlMin}_A(\pi)}} = 2^{n-1}\prod\limits_{i=1}^n{\left(\rh(q+i-1)+1-q\right)}.$$
\end{thm}

\end{document}